\documentclass[12pt,reqno]{amsart}
\usepackage{amssymb,amsthm}
\usepackage[utf8]{inputenc}

\advance\hoffset by -0.5truein
\let\<\langle
\let\>\rangle
\def\Var{\mathop {\fam0 Var}\nolimits}
\def\Pois{\mathrm{Pois}}
\def\Lie{\mathrm{Lie}}
\def\As {\mathop {\fam0 As }\nolimits}
\def\GD {\mathop {\fam0 GD }\nolimits}
\def\SGD {\mathop {\fam0 SGD }\nolimits}
\def\SNP {\mathop {\fam0 SNP }\nolimits}

\def\Com{\mathrm{Com}}
\def\Nov{\mathrm{Nov}}
\def\RNov{\mathrm{RNov}}
\def\BiCom{\mathrm{BiCom}}
\def\Der{\mathrm{Der}}
\def\wt{\mathop {\fam 0 wt}\nolimits}

\newtheorem{theorem}{Theorem}
\newtheorem{lemma}{Lemma}
\newtheorem{proposition}{Proposition}
\newtheorem{corollary}{Corollary}

\theoremstyle{definition}
\newtheorem{definition}{Definition}
\newtheorem{remark}{Remark}
\newtheorem{example}{Example}

\title[Gelfand--Dorfmann algebras]{Gelfand--Dorfman algebras, derived identities, and the Manin product of operads}
\author[P.S. Kolesnikov, B. Sartayev, A. Orazgaliev]{P.S. Kolesnikov$^{1)}$, B. Sartayev$^{2)}$, A. Orazgaliev$^{3)}$}

\address{${}^{1)}$ Sobolev Institute of Mathematics, Novosibirsk, Russia}
\address{$^{2)}$ Suleyman Demirel University, Kaskelen, Kazakhstan}
\address{$^{3)}$ Al-Farabi Kazakh National University, Almaty, Kazakhstan}

\begin{document}

\begin{abstract}
Gelfand--Dorfman bialgebras (GD-algebras) are nonassociative systems with two bilinear operations satisfying a series of identities 
that express Hamiltonian property of an operator in the formal calculus of variations.
The paper is devoted to the study of GD-algebras related with differential Poisson algebras. As a byproduct, 
we obtain a general description of identities that hold for operations $a\succ b  = d(a)b$ and $a\prec b = ad(b)$ 
on a (non-associative) differential algebra with a derivation~$d$.
\end{abstract}


\maketitle

\section{Introduction}

A series of non-associative structures related with the formal calculus of variations appeared in 
\cite{GD83}. These structures are nowadays known as Novikov algebras  
(they were independently 
introduced in \cite{BN85} as a tool for classification of linear Poisson brackets of hydrodynamic type), 
Novikov--Poisson algebras, and Gelfand--Dorfman bialgebras. 
The latter class of systems was initially invented in \cite{GD83} as a source of Hamiltonian operators: the structure constants of a Gelfand--Dorfman bialgebra may be used as coefficients of a differential operator.  

By definition, 
a linear space $A$ equipped with one bilinear operation 
$\circ: A\times A\rightarrow A$ is said to be a {\em Novikov algebra}
(the term was proposed in \cite{Osborn}) if it satisfies 
the following identities:
\begin{gather}
(a\circ b)\circ c-a\circ (b\circ c)=(b\circ a)\circ c-b\circ (a\circ c), \label{eq:LSymm} \\
(a\circ b)\circ c=(a\circ c)\circ b. \label{eq:RComm}
\end{gather}
A {\em Gelfand---Dorfmann bialgebra} \cite{Xu2000} is a system  $(A,\circ,[\cdot,\cdot])$
with two bilinear operations such that $(A,\circ)$ is a Novikov algebra, $(A,[\cdot,\cdot])$ 
is a Lie algebra, and the following additional identity holds:
\begin{equation}\label{eq:GD-1} 
[a,b\circ c]-[c,b\circ a]+[b,a]\circ c-[b,c]\circ a-b\circ[a,c]=0.
\end{equation}
In order to avoid a collision with the well-known notion of a bialgebra as an algebra equipped
with a coproduct, we will use the term {\em GD-algebra} for a Gelfand---Dorfman bialgebra.

It turned out \cite{Xu2000} (see also \cite{HongWu17}) that GD-algebras 
are closely related with conformal Lie algebras appeared in the theory of vertex operators \cite{Kac1998}. 
Conformal algebras and their generalizations (pseudo-algebras)
also turn to be useful 
for the classification of Poisson brackets of hydrodynamic type \cite{BDK}. 
It worth mentioning that Poisson algebras are also closely related with 
representations of Lie conformal algebras \cite{Kol:PoisPrepr}. 

The following statement provides a series of examples of GD-algebras. 
Recall that a Poisson algebra is a linear space $V$ equipped with two bilinear operations, i.e., a system  
$(V,\cdot,\{\cdot,\cdot\})$, where $(V,\cdot)$ is an associative and commutative algebra, 
$(V,\{\cdot,\cdot\})$ is a Lie algebra, and the 
following {\em Leibniz identity} holds:
\[
\{x,y z\}=\{x,y\} z+y\{x,z\}.
\]

\begin{theorem}\label{thm:th1}
Let $V$ be a Poisson algebra with a derivation $d\in \mathrm{Der}\, V$ 
relative to the both products.
Define a new operation $\circ $ on $V$ in the following way:
\begin{equation}\label{eq:Pois_to_GD}
 x\circ y=x d(y).
\end{equation}
Then $(V, \circ , \{\cdot , \cdot \})$ is a GD-algebra.
\end{theorem}

The proof is straightforward.
It is well-known \cite{BN85, GD83} that $x\circ y=x d(y)$ satisfies \eqref{eq:LSymm} and \eqref{eq:RComm}. It remains to show that 
\eqref{eq:GD-1} holds. Indeed, evaluate the left-hand side of \eqref{eq:GD-1} 
in a differential Poisson algebra:
$[u,w\circ v]-[v,w\circ u]+[w,u]\circ v-[w,v]\circ u-w\circ[u,v]
=\{u,wd(v)\}-\{v,wd(u)\}+\{w,u\}d(v)-\{w,v\}d(u)-wd(\{u,v\})
=\{u,w\}d(v)+\{u,d(v)\}w-\{v,w\}d(u)-\{v,d(u)\}w+\{w,u\}d(v)-\{w,v\}d(u)-w\{d(u),v\}
-w\{u,d(v)\}=0$.\qed

Let us call a GD-algebra $A$ {\em special} if it can be embedded into a differential Poisson algebra 
with operations $\{\cdot , \cdot\}$ and $\circ $ given by \eqref{eq:Pois_to_GD}. 
The class of all homomorphic images of all special GD-algebras 
is a variety (i.e., a class of algebras which is closed with respect to homomorphic images, subalgebras, and 
Cartesian products). Let us denote this variety by $\SGD$. 
The main purpose of this paper is to describe a way
for finding the defining identities of $\SGD$. 

In order to solve that problem, we need to study differential algebras and 
answer the following question which is of independent interest. 
Suppose $A$ is a (non-associative) algebra with a derivation~$d$. Introduce the following 
new binary operations on $A$: 
\begin{equation}\label{eq:SucPrec}
 a\succ b = d(a)b, \quad a\prec b = a d(b), \quad a,b\in A.
\end{equation}
Assuming $A$ satisfies a family of polynomial identities, what can be said about identities on 
$A^{(d)}=(A, \succ, \prec )$? 

For example, if $A$ is associative and commutative then 
$a\succ b = b\prec a$ and the operation $\prec $ satisfies \eqref{eq:LSymm} and \eqref{eq:RComm}. 
It was shown in \cite{DzhLowf} that there are no more independent identities that hold on $A^{(d)}$ 
for all associative commutative algebras $A$ and for all  their derivations~$d$.
Moreover, it was proved in \cite{BCZ-NP} that every Novikov algebra may be embedded into 
a differential associative commutative algebra. 

We find an explicit answer to the question about identities of $A^{(d)}$. 
It turns out 
that, given a variety $\Var $ defined by multilinear identities (multilinear variety, for short),
 the systems $A^{(d)}$, $A\in \Var $, $d\in \Der (A)$, generate a variety 
governed by the operad $\Var\circ \Nov$, 
where $\Nov $ is the operad of Novikov algebras and $\circ $ stands for the Manin white product of 
operads \cite{GK94, LodVal08}. Hereinafter we identify notations for multilinear varieties and the 
corresponding operads.

If $\Var $ is a variety of algebras 
then the free algebra in $\Var $ generated by a set $X$ is denoted 
$\Var \<X\>$. By $\Var\Der $ we denote the class of  all differential $\Var $-algebras with one derivation (this is a variety of algebraic systems 
with one additional unary operation $d$ in the language).
The free algebra in $\Var\Der $ generated by a set $X$ 
is denoted $\Var\Der\<X,d\>$.
We will often need to know the structure of the free differential $\Var $-algebra. 
It can be easily derived in general (c.f. \cite{GuoKeig08} for associative algebra case) that the following statement holds.

\begin{lemma}\label{lem:VarDerFree}
For a multilinear variety $\Var$, the
free differential $\Var$-algebra generated by a set $X$ 
with one derivation $d$ coincides with
$\Var\<d^\omega X\>$, where $d^\omega X = \{d^s(x) \mid s\ge 0, x\in X\}$.
\end{lemma}

Indeed, consider the free nonassociative (magmatic) algebra $M\<d^\omega X\>$ and define a linear map 
$d: F(X;d)\to F(X;d)$ in such a way that $d(d^s(x))=d^{s+1}(x)$, $d(uv)=d(u)v+ud(v)$.
Denote by $T_{\Var}(X)$ 
the T-ideal of all identities in a set of variables $X$ that hold on $\Var $. 

Since $\Var $ is defined by multilinear identities, the T-ideal $T_{\Var} (d^\omega X)$ of $M\< d^\omega X\>$
is $d$-invariant:
\[
 d(f(d^{s_1}(x_1),\dots, d^{s_n}(x_n)) = \sum\limits_{i=1}^n f(d^{s_1}(x_1),\dots , d^{s_i+1}(x_i), \dots , d^{s_n}(x_n)).
\]
Hence, $U=F(X;d)/T_{\Var}(d^\omega X)$ is a $\Var $-algebra with a derivation~$d$.
Since all relations from $T_{\Var}(d^\omega X)$ hold in every differential $\Var$-algebra, 
$U$ is isomorphic to the free 
differential $\Var$-algebra. \qed

\begin{remark}
In Lemma \ref{lem:VarDerFree}, a derivation $d$ 
may be replaced with a {\em generalized derivation}, 
i.e., a linear map $D$ such that 
\[
D(xy)=D(x)y+xD(y)+\lambda xy,
\]
where $\lambda $ is a fixed scalar in $\Bbbk $. 
Indeed, it is enough to note that
 $d=D-\lambda \mathrm{id}$ is an ordinary derivation.
\end{remark}

In order to study the relations between GD-algebras and differential Poisson algebras we need 
more general settings. Namely, given an algebra $A$ with a derivation $d$, consider 
the system $A^{(d*)}=(A, \succ, \prec , \cdot )$, where $\succ $, $\prec $ are defined 
by \eqref{eq:SucPrec}, and $\cdot $ is the initial multiplication on~$A$.
How to describe the identities that hold on $A^{(d*)}$ provided that we know 
the identities of $A$? 
It turns out that the answer involves the operad $\GD$ of GD-algebras. 
Namely, 
for a multilinear variety $\Var $, the systems 
$A^{(d*)}$, $A\in \Var $, $d\in \Der (A)$, generate a variety governed by 
$\Var \circ \GD^!$, where $\GD^!$ is the Koszul-dual of the operad $\GD $.


The paper is organized as follows. 
In Section \ref{sec:DualGD}, 
we calculate the operad 
$\GD^!$ and describe the free $\GD^!$-algebra. 
In fact, the variety of $\GD^!$-algebras is a proper sub-variety of the class of Novikov--Poisson algebras. 
The defining identities of $\GD^!$ were earlier obtained in \cite{BCZ-NP}, 
but it was assumed in \cite{BCZ-NP} that a Novikov--Poisson algebra has an identity element. 
We can not use this assumption for our purpose, so we describe the free $\GD^!$-algebra 
explicitly and prove that it can be embedded into
the free differential associative and commutative algebra.

In Section \ref{sec:Derived}, we define (generalized) derived  identities of a multilinear variety $\Var $ 
of algebras with binary operations as those identities that hold on all $A^{(d)}$ (resp., $A^{(d*)}$) 
for all $A\in \Var $ and for all $d\in \Der (A)$. Denote the operads defined by (generalized) derived 
identities by $D\Var $ (resp., $D^*\Var $). Then we prove 
\[
 D\Var = \Var\circ \Nov,\quad D^*\Var = \Var \circ \GD^!.
\]

In Section \ref{sec:SpecialGD}, we describe the linear basis of the free special GD-algebra $\SGD\<X\>$
in terms of the enveloping differential Poisson  algebra. This result generalizes what was done 
in \cite{DzhLowf}: here we have to add one more operation $[\cdot , \cdot ]$ 
into consideration. 

Section \ref{sec:Ident} is devoted to the description of defining relations of the operad SGD governing 
the variety generated by all special GD-algebras.
It is easy to see that the operad SGD is a sub-operad of $D^*\Pois $ generated by 
a proper subspace of binary operations. We find the entire family of defining relations 
for $D^*\Pois $, so all defining identities of SGD occur among their corollaries. 
In particular, we are interested in 
{\em special identities} of GD-algebras, i.e., 
those identities that hold on  all SGD-algebras but do not hold on the class of all GD-algebras.
There are no special identities of degree 3, but we can find two independent special 
identities of degree 4. The complete list of special identities of GD-algebras remains
unknown.

\section{Dual Gelfand---Dorfmann algebras}\label{sec:DualGD}

Suppose $\Var $ is a multilinear variety of algebras.
Fix a countable set of variables 
$X=\{x_1,x_2,\dots \}$ and denote 
\[
 \Var(n) = M_n(X)/M_n(X)\cap T_{\Var}(X), 
\]
where $M_n(X)$ is the space of all (non-associative) multilinear polynomials of degree $n$ 
in $x_1,\dots, x_n$, $T_{\Var}(X)$ is the T-ideal of all identities that hold in $\Var $. 

The collection of spaces $(\Var(n))_{n\ge 1}$ forms a symmetric operad relative to the natural 
composition rule and symmetric group action (see, e.g., \cite{LodVal08}). We will denote this operad by the same symbol $\Var $. 
The operad $\Var $ is a {\em binary} one, i.e., it is generated (as a symmetric operad) 
by the elements of $\Var(2)$. In the generic case of algebras with one binary operation of multiplication, 
$\Var(2)$ is spanned by $x_1x_2$ and $x_2x_1$.

The variety of GD-algebras is defined by identities of degree 2 or 3. 
Therefore, the corresponding operad is a 
{\em quadratic} one \cite{GK94}. 
Let us compute the Koszul dual operad $\GD^!$.

Denote two operations of a GD-algebra by $\mu$ and $\nu$, namely,
$\mu =[x_1 , x_2]$, 
$\nu = x_1 \circ x_2$.
Then 
$\GD(2) = \mathrm{Span}\{\mu,\nu,\nu^{(12)}\}$, 
where $\nu_2^{(12)} = x_2\circ x_1$.

As in \cite{GK94}, the $S_3$-space of multilinear terms of degree 3 is identified with
\[
 O_3 = \Bbbk S_3\otimes _{\Bbbk S_2}(\GD(2)\otimes \GD(2)),
\]
where the action of $(12)\in S_2 $ on $\GD(2)^{\otimes 2}$ is given by
$\text{id} \otimes (12)$.

We may choose a basis $a_1,\dots , a_{27}$ of $O_3$
as follows: 
\[
 \begin{aligned}
& a_1=1\otimes_{\Bbbk S_2}(\mu\otimes \mu) = [[x_1, x_2], x_3], \quad
&& a_2=1\otimes_{\Bbbk S_2} (\mu\otimes \nu) =[x_1\circ x_2, x_3], \\
& a_3= 1\otimes_{\Bbbk S_2}\big (\mu\otimes \nu^{(12)}\big ) =[x_2\circ x_1 , x_3], \quad
&& a_4=1\otimes_{\Bbbk S_2} (\nu\otimes \mu) = [x_1, x_2] \circ x_3, \\
& a_5 = 1\otimes_{\Bbbk S_2} (\nu\otimes \nu)= (x_1\circ x_2)\circ x_3, \quad
&& a_6= 1\otimes_{\Bbbk S_2}\big (\nu\otimes \nu^{(12)}\big )=(x_2\circ x_1)\circ x_3, \\
& a_7 = 1\otimes_{\Bbbk S_2}\big (\nu^{(12)}\otimes \mu\big ) = x_3\circ [x_1, x_2], \quad
&& a_8 = 1\otimes_{\Bbbk S_2}\big (\nu^{(12)}\otimes \nu\big ) = x_3\circ (x_1\circ x_2), \\
& a_9 = 1\otimes_{\Bbbk S_2}\big (\nu^{(12)}\otimes \nu^{(12)}\big ) = x_3\circ (x_2\circ x_1), && \\
\end{aligned}
\]
and $a_{9+i}= (13)a_i$, $a_{18+i}=(23)a_i$ for $i=1,\dots, 9$.

Relative to the basis $a_1,\dots, a_{27}$, the vectors representing the defining relations of $\GD $
may be presented by the rows of the following matrix:
\[
\begin{smallmatrix}
 1 &0 &0 & 0 &0 & 0 &0 &0 & 0 &-1 &0 & 0 &0 &0 & 0 &0 & 0 & 0 &-1 & 0 & 0 & 0 & 0 & 0 & 0 &0 &0 \\
 0 &0 &0 & 0 &1 & 0 &0 &0 & 0 & 0 &0 & 0 &0 &0 & 0 &0 & 0 & 0 & 0 & 0 & 0 & 0 &-1 & 0 & 0 &0 &0 \\
 0 &0 &0 & 0 &0 & 1 &0 &0 & 0 & 0 &0 & 0 &0 &0 &-1 &0 & 0 & 0 & 0 & 0 & 0 & 0 & 0 & 0 & 0 &0 &0 \\
 0 &0 &0 & 0 &0 & 0 &0 &0 & 0 & 0 &0 & 0 &0 &1 & 0 &0 & 0 & 0 & 0 & 0 & 0 & 0 & 0 &-1 & 0 &0 &0 \\
 0 &0 &0 & 0 &1 &-1 &0 &0 & 0 & 0 &0 & 0 &0 &0 & 0 &0 & 0 &-1 & 0 & 0 & 0 & 0 & 0 & 0 & 0 &1 &0 \\
 0 &0 &0 & 0 &0 & 0 &0 &1 & 0 & 0 &0 & 0 &0 &0 & 0 &0 &-1 & 0 & 0 & 0 & 0 & 0 & 1 &-1 & 0 &0 &0 \\
 0 &0 &0 & 0 &0 & 0 &0 &0 &-1 & 0 &0 & 0 &0 &1 &-1 &0 & 0 & 0 & 0 & 0 & 0 & 0 & 0 & 0 & 0 &0 &1 \\
 0 &1 &0 & 1 &0 & 0 &0 &0 & 0 & 0 &0 & 0 &0 &0 & 0 &1 & 0 & 0 & 0 &-1 & 0 &-1 & 0 & 0 & 0 &0 &0 \\
 0 &0 &1 &-1 &0 & 0 &0 &0 & 0 & 0 &0 &-1 &1 &0 & 0 &0 & 0 & 0 & 0 & 0 & 0 & 0 & 0 & 0 &-1 &0 &0 \\
 0 &0 &0 & 0 &0 & 0 &1 &0 & 0 & 0 &1 & 0 &1 &0 & 0 &0 & 0 & 0 & 0 & 0 &-1 & 1 & 0 & 0 & 0 &0 &0 
\end{smallmatrix}
\]
Compute the sign-twisted orthogonal complement $U^\perp \subset O_3^\vee $ \cite{GK94}
for the $S_3$-module $U$ spanned by these vectors 
to obtain the defining relations of the Koszul dual operad $\GD^!$.
It turns out that a $\GD^!$-algebra has two operations
$\mu^\vee(x_1,x_2) = x_1\ast x_2$ and  $\nu^\vee (x_1,x_2) = x_1\star x_2 $, 
where 
\begin{equation}\label{eq:NP-1}
 x\ast y = y\ast x, \quad 
 x\ast (y\ast z) = (x\ast y)\ast z ,
\end{equation}
i.e., 
$\ast $ is associative and commutative, 
\begin{gather}
 (x\star y)\star z - x\star ( y\star z)  = (x\star z)\star y - x\star ( z\star y), \label{eq:NP-2}\\
 x\star ( y\star z) = y\star ( x\star z), \label{eq:NP-3} 
\end{gather}
i.e., $\star $ is a right Novikov product, and 
\begin{gather}
x\star(y\ast z)=(x\star y)\ast z,  \label{eq:21}\\
x\star(y\ast z)+y\star(x\ast z)=(x\ast y)\star z. \label{eq:3}
\end{gather}

\begin{remark}
Relations \eqref{eq:NP-1}--\eqref{eq:3} up to a change of notations 
\[
 xy = x\ast y,\quad x\circ y = y\star x
\]
coincide with the  identities defining a {\em differential GDN-Poisson algebra} \cite{BCZ-NP}. 
This is a proper sub-variety of the variety of Novikov---Poisson algebras 
introduced in \cite{Xu1997}.
However, all algebras in \cite{BCZ-NP} are assumed to be unital, i.e., with an identity 
$e$ relative to the associative and commutative operation $\ast $. This is an essential restriction 
since it is unclear in general how to join a unit to a $\GD^!$-algebra. This is why in this section 
we have to study the free $\GD^!$-algebra without an identity assumption. 
\end{remark}

Let us find a normal form of elements in the free $\GD^!$-algebra, i.e., 
a linear basis of $\GD^!\<X\>$.

Denote by $\BiCom$ the variety of algebras with two 
associative and commutative operations $*$ and $\odot $
such that 
\begin{equation}\label{eq:BicomIdentity}
(x\odot y)* z = x\odot (y* z).    
\end{equation}

It is easy to see that the following identities hold in $\BiCom$:
\begin{equation}\label{eq:BiComCorollary}
 (x\odot y)* z = (x* y)\odot z,
 \quad 
 x* (y\odot z) = y* (x\odot z).
\end{equation}
Indeed, apply commutativity of $\ast $ and \eqref{eq:BicomIdentity} to obtain
\begin{multline}\nonumber
(x\odot y)* z=(y\odot x)* z=y\odot (x* z)=y\odot (z* x)=(y\odot z)* x \\
=(z\odot y)* x=z\odot (y* x)=(x* y)\odot z.
\end{multline}
Similarly, 
\[
x* (y\odot z)=(z\odot y)*x=z\odot (y*x)=z\odot (x*y)=(z\odot x)*y=y* (x\odot z).
\]

\begin{theorem}\label{thm:BiCom-Basis}
A linear basis of the free $\BiCom$-algebra 
$\BiCom\<X\>$ generated by a linearly ordered 
set $X$ consists of the words
\begin{equation}\label{eq:BiCom-Basis}
 x_1* \dots * x_t * (y_1\odot \dots \odot y_k),
 \quad 
 x_1\le \dots \le x_t\le y_1\le \dots \le y_k,
\end{equation}
$x_i,y_j\in X$.
\end{theorem}

\begin{proof} 
It can be easily seen that every element of a $\BiCom$ algebra
generated by a set $X$ can be transformed to 
a linear combination of words that have the following form:
\[
x_{i_1}\ast x_{i_2}\ast\ldots\ast x_{i_t}\ast(y_{j_{1}}\odot\ldots\odot y_{j_k}),\quad x_i,y_j\in X.
\]
Commutativity of $\ast$ and $\odot$ allows us to assume 
$x_{i_1}\le \dots \le x_{i_t}$ and $y_{j_1}\le \dots \le y_{j_k}$. 
If $x_{i_t}>y_{j_1}$ then we may interchange them by means of the second 
relation in \eqref{eq:BiComCorollary}.
Hence, every element of $\BiCom$ can be transformed to
a linear combination of monomials (\ref{eq:BiCom-Basis}).

To prove linear independence of (\ref{eq:BiCom-Basis}),
 consider a linear space $B'$ spanned by (\ref{eq:BiCom-Basis}) 
 and define multiplications $\ast$ and $\odot $ on $B'$ in the following way.
 If $u=x_1\ast\ldots\ast x_n\ast(y_1\odot\ldots\odot y_m)$,
 $v=z_1\ast\ldots\ast z_p\ast(t_1\odot\ldots\odot t_q)$
 then 
\[
u\ast v =
l_1\ast\ldots\ast l_{n+p+1}\ast(l_{n+p+2}\odot\ldots\odot l_{n+p+m+q})
\]
\[
u\odot v = 
l_1\ast\ldots\ast l_{n+p}\ast(l_{n+p+1}\odot\ldots\odot l_{n+p+m+q})
\]
where the sequence $(l_1,\ldots,l_{n+p+m+q})$ 
is just the ordering of the sequence
$(x_1,\dots, x_n,$ $y_1,\dots , y_m, z_1,\dots , z_p, t_1,\dots , t_q)$.
It is easy to check that $B^{'}$ is a $\BiCom$ algebra. Hence, 
$B^{'}\simeq  \BiCom\<X\>$.
\end{proof}

Let $\BiCom\Der\<X,d\>$ stand for the free differential 
$\BiCom$-algebra generated by a set $X$ with a derivation~$d$
(i.e., $d$ is a derivation relative to the both 
operations $\ast $ and~$\odot $).

\begin{theorem}\label{thm:BiComDer-GDual}
Let $A$ be a $\BiCom$-algebra with a derivation $d$. 
Then $A$ is a $\GD^!$-algebra relative to the operations $*$ and $\star$
defined by
\[
x\star y =d(x)\odot y.
\]
\end{theorem}

\begin{proof}
The operation $x\star y =d(x)\odot y$ is known to be right Novikov. 
It remains to check the identities (\ref{eq:21}), (\ref{eq:3}):
\[
x\star(y\ast z)-(x\star y)\ast z=d(x)\odot(y\ast z)-(d(x)\odot y)\ast z=0,
\]
\begin{multline}\nonumber
x\star(y\ast z)+y\star(x\ast z)=d(x)\odot(y\ast z)+d(y)\odot(x\ast z)
=
(d(x)\odot y)\ast z+(d(y)\odot x)\ast z  \\
=
(d(x)\ast y)\odot z+(d(y)\ast x)\odot z
=
d(x\ast y)\odot z 
=
(x\ast y)\star z
\end{multline}
by~\eqref{eq:BiComCorollary}.
\end{proof}

Let us say that a $\GD^!$-algebra $B$ is {\em special} if it can be 
embedded into an appropriate differential $\BiCom$-algebra. 
Denote by $\SGD^{!}$ the variety generated by special $\GD^!$-algebras
(all homomorphic images of special $\GD^!$-algebras).
Obviously, the free $\SGD^{!}$-algebra $\SGD^!\<X\>$ generated by a set 
$X$ is isomorphic to the subalgebra of $\BiCom\Der\<X,d\>$ 
generated by $X$ relative to the operations $\star$ and~$*$.

For every monomial $u\in \BiCom\Der\<X,d\>$ define the weight 
of $u$ by the following induction rule:
\[
\begin{gathered}
 \wt(x)=-1,\quad x\in X; \\
 \wt(d(u)) = \wt(u)+1; \\
 \wt(u* v) = \wt(u)+\wt(v)+1;\\
 \wt(u\odot v) = \wt(u)+\wt(v).
\end{gathered}
\]
The space of differential $\BiCom$-monomials of weight $-1$ 
has the following basis:
\begin{equation}\label{eq:SDG!-Basis}
\begin{gathered}
x_1*\dots * x_t *(y_1\odot \dots \odot y_i \odot d^{r_{i+1}}(y_{i+1}) \odot \dots \odot d^{r_k}(y_k)), \\
x_1\le \dots \le x_t\le y_1\le \dots \le y_i, \\
r_{i+1},\dots , r_k>0,\quad r_{i+1}+\dots +r_k-k=-1, \\
(r_{i+1},y_{i+1})\le \dots \le (r_k,y_k) \ \text{lexicographically}, 
\end{gathered}
\end{equation}

\begin{theorem}\label{thm:SDG!-weight}
An element $f\in \BiCom\Der\<X,d\>$ belongs to 
$\SGD^!\<X\>$ if and only if all monomials that appear in $f$ have weight~$-1$.
\end{theorem}

\begin{proof}
The ``only if''
 part can be easily proved by induction. For the converse (``if'') part, it is enough to consider 
the case when $f$ is a monomial of the form \eqref{eq:SDG!-Basis}. 
Note that if $\wt(f)=-1$ then
$f=x_1*\dots * x_t*w$, where
$w = y_1\odot \dots \odot y_i \odot 
d^{r_{i+1}}(y_{i+1})\odot \dots \odot d^{r_k}(y_k)$,  
$\wt(w)=-1$.
As shown in \cite{DzhLowf}, $w$
can be expressed via $\star $. 
Then $f\in \SGD^!\<X\>$.
\end{proof}

\begin{theorem}\label{thm:SGD!=GD!}
For a linearly ordered set $X$, $\SGD^!\<X\>$ is isomorphic to $\GD^!\<X\>$.
\end{theorem}

\begin{proof}
First, let us prove the following technical statement.
\begin{lemma}\label{lem:GD!-Normal}
The following elements span the free $\GD^!$-algebra 
$\GD^!\<X\>$ generated by a linearly ordered set $X$:
\[
 x_1*\dots *x_t*w,\quad x_1\le \dots \le x_t, \ x_i\in X,
\]
where $w$ is a monomial from $\RNov \<X\>$ relative to the operation~$\star $.
\end{lemma}

Here $\RNov $ denotes the variety of right Novikov algebras defined by 
the identities \eqref{eq:NP-2} and \eqref{eq:NP-3}.

\begin{proof}
Let $f$ be a monomial in $\GD^!\<X\>$. Proceed by induction on $n=\deg f$.
If $n=1,2$ then the statement is clear. 
Suppose $n\ge 3$. Then $f=u*v$ or $f=u\star v$, 
$\deg u,\deg v <n$.
By the inductive assumption, 
\[
u = x_1*\dots x_r * w_1,
\quad 
v = y_1*\dots y_s * w_2,
\]
where 
$w_1,w_2\in \RNov\<X\> $.

Case 1: $f=u*v$. Then $f=x_1*\dots *x_r*y_1*\dots *y_s * (w_1*w_2)$.
If $\deg w_1=1$ or $\deg w_2=1$ then we are done.
If $w_i=(w_{i1}\star w_{i2})$, $i=1,2$, then the inductive assumption implies 
\[
(w_{11}\star w_{12})*(w_{21}\star w_{22} ) = 
w_{11}\star (w_{12}*(w_{21}\star w_{22} )) = 
w_{11}\star (z*w')=z*(w_{11}\star w'),
\]
where $z\in X$, $w'\in \RNov\<X\>$.
Hence, 
$f=x_1*\dots *x_r*y_1*\dots *y_s * z*(w_{11}\star w')$.

Case 2: $f=u\star v$. 
Then $f=(x_1*\dots *x_r*w_1)\star(y_1*\dots *y_s * w_2)$.
If $\deg u=1$ then the statement follows from \eqref{eq:21}
and commutativity of $*$:
\[
f=x_1\star (y_1*\dots *y_s * w_2) = y_1*\dots * y_r * (x_1\star w_2).
\]
For $\deg u>1$, we have to consider the following sub-cases:
$r=0$ and $r>0$. The first one is similar to what is done for $\deg u=1$, 
in the second we have
$u=x_1*u'$, so by \eqref{eq:3}
\[
f=(x_1*u')\star (y_1*\dots *y_s * w_2) = 
((x_1\star u') + (u'\star x_1))*(y_1*\dots *y_s * w_2).
\]
The latter expression is of the type considered in Case~1.
\end{proof}

Suppose $X$ is linearly ordered. 
Consider the epimorphism 
\[
\varphi : \GD^!\<X\> \to \SGD^!\<X\>\subset \BiCom\Der \<X,d\>
\]
given by $\varphi(x)=x$ for $x\in X$, 
$\varphi(u\star v) = d(\varphi(u))\odot \varphi(v)$, 
$\varphi(u*v) = \varphi(u)*\varphi(v)$.

Recall that elements from $\RNov \<X\>$ may be presented as differential 
polynomials from the free differential commutative algebra
$\Com\Der\<X,d\>$
(relative to binary operation $\odot $ and derivation $d$)
 by means of $x\star y = d(x)\odot y$.
Such a presentation is known to be unique \cite{DzhLowf}.
Therefore, by Lemma~\ref{lem:GD!-Normal}
$\GD^!\<X\>$ is spanned by elements
\begin{equation}\label{eq:GD!-Normal}
 u = x_1*\dots *x_t*(d^{s_1}(y_1) \odot \dots \odot d^{s_k}(y_k)),
 \quad x_i\in X, \ y_j\in X,
\end{equation}
where 
$x_1\le \dots \le x_t$,
$s_1+\dots +s_k-k=-1$, 
and 
\[
(s_1,y_1)\le \dots \le (s_k,y_k)\ \text{lexicographically}.
\]
Note that if $u$ is of the form \eqref{eq:GD!-Normal} then 
$\varphi(u)=u \in \BiCom\Der\<X,d\>$.
Moreover, $s_1=0$.

Relation \eqref{eq:BiComCorollary} implies that 
 we may add one more condition on the words 
\eqref{eq:GD!-Normal}: $x_t\le y_1$. 

Therefore, we have found a set of linear generators in 
$\GD^!\<X\>$ whose images in $\SGD^!\<X\>$ are linearly independent. 
Hence, 
$\GD^!\<X\> \simeq \SGD^!\<X\>$.
In particular, the set of monomials \eqref{eq:SDG!-Basis} is a linear basis of $\GD^!\<X\>$.
\end{proof}

Consider the class $\SNP$ of all algebras with two operations 
$\ast $ and $\star $ obtained from  differential associative commutative algebras 
via
\begin{equation}\label{eq:ComDerOp}
a\ast b = ab, \quad a\star b = d(a)b.
\end{equation}

\begin{corollary}\label{cor:GD-SNP}
The class of all homomorphic images of algebras from $\SNP$ coincides with the variety of $\GD^!$-algebras.
\end{corollary}

In other words, the defining identities of $\GD^!$ and their corollaries exhaust all identities that hold for operations
\eqref{eq:ComDerOp} on commutative differential algebras.

\begin{proof}
Consider the homomorphism 
of $\GD^!$-algebras 
\[
\psi : \GD^!\<X\> \to \Com\Der\<X,d\>
\]
defined as follows:
$\psi(x)=x$ for $x\in X$,
$\psi(f*g)=\psi(f)\psi(g)$,
$\psi(f\star g)= d(\psi(f))\psi(g)$.
We are interested in the kernel of $\psi $.

By Theorem~\ref{thm:SGD!=GD!}, $\GD^!\<X\>\subset \BiCom\Der\<X,d\> $.
Therefore, $\psi$ is the restriction of the following 
homomorphism of differential $\BiCom$-algebras:
\[
\varphi: \BiCom\Der\<X,d\> \to \Com\Der\<X,d\>,
\]
$\varphi(x)=x$ for $x\in X$,
$\varphi(f*g)=\varphi(f\odot g)=\varphi(f)\varphi(g)$.

Suppose $X$ is equipped with a linear order and $w$ is a monomial in $\BiCom\Der\<X,d\>$ of the form 
\eqref{eq:SDG!-Basis}. Then 
\[
\varphi(w) = x_1\dots x_ty_1\dots y_id^{r_{i+1}}(y_{i+1})\dots d^{r_k}(y_k),
\]
where
$x_1\le \dots \le x_t \le y_1\le \dots \le y_i$,
$r_{i+1}>0$,
$(r_{i+1},y_{i+1})\le \dots \le (r_k,y_k)$
lexicographically. 
Therefore, the images of the basic monomials of $\GD^!\<X\>$ are linearly independent, 
so 
\[
\mathrm{Ker}\,\psi = GD^!\<X\>\cap \mathrm{Ker}\,\varphi = \{0\}.
\]
\end{proof}

In particular, we may find a simple form for a basis of the free algebra $\GD^!\<X\>$.

Let $w$ be a monomial in $\Com\Der\<X,d\>$. Define the weight of $w$ by induction as follows:
 $\wt(x)=-1$, $x\in X$; $\wt(uv)=\wt(u)+\wt(v)$; $\wt(d(u))=\wt(u)+1$.

\begin{corollary}\label{eq:ComDerWeight}
A linear basis of $\GD^!\<X\>$ consists of all monomials in $\Com\Der\<X,d\>$ 
of weight $\le -1$.
\end{corollary}

\begin{remark}
For unital $\GD^! $-algebras, a more precise result was obtained in \cite{BCZ-NP}: 
if $V$ is a $\GD^!$-algebra with an identity 
 $e$ relative to the operation $*$ then $V$ coincides with an associative commutative algebra 
 $A$
 with a derivation $d\in \Der (A)$
 equipped with the operations 
 $x\ast y = xy$, $x\star y = d(x)y$. 
\end{remark}

\section{Derived identities}\label{sec:Derived}

Let $\mathcal P$ and $\mathcal Q$ be two  operads. Then the family of spaces
$(\mathcal P(n)\otimes \mathcal Q(n))_{n\ge 1}$ is an operad relative to 
the componentwise composition and symmetric group action.
This operad, denoted $\mathcal P\otimes \mathcal Q$, is 
known as the {\em Hadamard product} of $\mathcal P$ and $\mathcal Q$. 
If $\mathcal P$ and $\mathcal Q$ are binary operads then
$\mathcal P\otimes \mathcal Q$
 may  be non-binary. The sub-operad of $\mathcal P\otimes \mathcal Q$ generated by 
 $\mathcal P(2)\otimes \mathcal Q(2)$ is known as the {\em Manin white product} of $\mathcal P$ and $\mathcal Q$, 
  it is denoted $\mathcal P\circ \mathcal Q$ \cite{GK94}. 
 
\begin{example}\label{exmp:Lie-Nov}
 The operad $\Lie \circ \Nov$ is isomorphic to the operad governing the class of all 
 algebras ({\em magmatic} operad).
\end{example}

Indeed, both $\Lie $ and $\Nov $ are quadratic operads, and so is $\Lie\circ \Nov$ \cite{GK94}. 
Let $[x_1,x_2]$ and $x_1\circ x_2$ be the generators of $\Lie$ and $\Nov$. 
It is enough to find the defining identities of $\Lie \circ \Nov$ that are quadratic with 
respect to the generators. 

Identify $[x_1,x_2]\otimes (x_1\circ x_2)$ with $[x_1\prec x_2]$, then 
$[x_1,x_2]\otimes (x_2\circ x_1)
= ([x_2,x_1]\otimes (x_1\circ x_2))^{(12)}
= -[x_2\prec x_1]$. 
Hence, $(\Lie\circ \Nov)(n)$ is an image of the space $M_n(X)$ 
of all multilinear non-associative polynomials of degree $n$ in $X=\{x_1,x_2,\dots \}$ relative to the 
operation $[\cdot \prec \cdot ]$. Calculating the  compositions in the componentwise way, we obtain
\[
\begin{gathered}
m_1= [[x_1\prec x_2]\prec x_3] = [[x_1x_2]x_3]\otimes (x_1\circ x_2)\circ x_3, \\
m_2= [x_1\prec [x_2\prec x_3]] = [x_1[x_2x_3]]\otimes x_1\circ (x_2\circ x_3), 
\end{gathered} 
\]
It remains to find the intersection of the $S_3$-submodule generated by $m_1$ and $m_2$ 
in $M_3(X)\otimes M_3(X)$ with the kernel of the projection 
$M_3(X)\otimes M_3(X)\to \Lie(3)\otimes \Nov(3)$. Straightforward calculation
shows the intersection is zero. Hence, the operation $[\cdot \prec \cdot]$
satisfies no identities.

\begin{example}\label{exmp:As-Nov}
 The operad $\As \circ \Nov$ is generated by 4-dimensional space $(\As\circ \Nov)(2)$
 spanned by  
 \[
 \begin{aligned}
 (x_1\prec x_2) &= x_1x_2 \otimes x_1\circ x_2,  \quad
 (x_2\prec x_1) &= x_2x_1 \otimes x_2\circ x_1,  \\
 (x_1\succ x_2), &= x_1x_2 \otimes x_2\circ x_1,  \quad
 (x_2\succ x_1), &= x_2x_1 \otimes x_1\circ x_2,  \\
 \end{aligned}
 \]
 relative to the following identities:
 \begin{gather}
  (x_1\succ x_2)\prec x_3 - x_1\succ (x_2\prec x_3)=0, \label{eq:L-alg} \\
  (x_1\prec x_2)\prec x_3 -x_1\prec (x_2\succ x_3)+ (x_1\prec x_2)\succ x_3 - x_1\succ (x_2\succ x_3)=0.
								  \label{eq:LodDer}
 \end{gather}
\end{example}

This result may be checked with a straightforward computation. 
Namely, one has to find the intersection of the $S_3$-submodule in $M_3(X)\otimes M_3(X)$
generated by 
\[
\begin{gathered}
 (x_1\prec x_2)\prec x_3 = (x_1x_2)x_3\otimes (x_1x_2)x_3, \quad
 x_1\succ (x_2\succ x_3) = x_1(x_2x_3)\otimes (x_3x_2)x_1, \\
 (x_1\succ x_2)\prec x_3 = (x_1x_2)x_3\otimes (x_2x_1)x_3, \quad
 x_1\succ (x_2\prec x_3) = x_1(x_2x_3)\otimes (x_2x_3)x_1, \\
 (x_1\prec x_2)\succ x_3 = (x_1x_2)x_3 \otimes x_3(x_1x_2), \quad
 x_1\prec (x_2\succ x_3) = x_1(x_2x_3)\otimes x_1(x_3x_2), \\
 x_1\prec (x_2\prec x_3) = x_1(x_2x_3)\otimes x_1(x_2x_3), \quad
 (x_1\succ x_2)\succ x_3 = (x_1x_2)x_3\otimes x_3(x_2x_1)
\end{gathered}
\]
with the kernel of the projection $M_3(X)\otimes M_3(X)\to \As(3)\otimes \Nov(3)$.

Given a (non-associative) algebra $A$ with a derivation $d$, denote by 
$A^{(d)}$ the same linear space $A$ considered as a system with two binary 
linear operations of multiplication  $\prec $ and $\succ $ defined by
\[
 x\prec y = xd(y), \quad x\succ y = d(x)y,\quad x,y\in A.
\]
In the same settings, denote by $A^{d*}$ the system $(A,\prec, \succ, \cdot)$, 
where $\cdot $ is the initial multiplication on~$A$.

\begin{definition}\label{defn:Derived}
Let $\Var $ be a multilinear variety of algebras.
A non-associative polynomial $f(x_1,\dots, x_n)$ in two operations
of multiplication $\prec $ and $\succ $ is called 
a {\em derived identity} of $\Var $ if for every $A\in \Var $ and for every derivation 
$d\in \Der (A)$ the algebra $A^{(d)}$ satisfies the identity $f(x_1,\dots ,x_n)=0$.
A {\em generalized derived identity} is a polynomial in three operations 
$\prec$, $\succ $, and $\cdot $ that turns into an identity 
on  $A^{d*}$ for every $A\in \Var $ and for every $d\in \Der (A)$.
\end{definition}

Obviously, the set of all (generalized) derived identities is a T-ideal of the algebra of non-associative 
polynomials in two (three) operations.

For example,  $(x\succ y)=(y\prec x)$ is a derived identity of $\Com $. 
Moreover, the operation $(\cdot \prec \cdot )$ satisfies
the axioms of Novikov algebras \eqref{eq:LSymm} and \eqref{eq:RComm}. 
It was actually shown in \cite{DzhLowf} that the entire T-ideal of derived 
identities of $\Com $ is generated by these identities. Similarly, 
Corollary \ref{cor:GD-SNP} implies that 
the identities of a $\GD^!$-algebra \eqref{eq:NP-1}--\eqref{eq:3}
present the complete list of generalized derived identities of 
$\Com $ relative to the following change of notations:
\[
 x*y = xy, \quad x\star y = x\succ y = y\prec x.
\]

\begin{remark}
For the variety of associative algebras, it was mentioned in 
\cite{Lod10} that \eqref{eq:L-alg} and \eqref{eq:LodDer} are derived identities of~$\As $.
\end{remark}

Our aim in this section is to show that (generalized) derived identities of $\Var $ 
are exactly the same as defining relations of $\Var \circ \Nov$ ($\Var \circ \GD^!$).
We will focus on the proof of the description of generalized derived identities.

Let $\mathcal N_{\Var }$ be the class of all differential $\Var $-algebras with one locally nilpotent 
derivation, i.e., for every $A\in \mathcal N_{\Var}$ and for every $a\in A$ there 
exists $n\ge 1$ such that $d^n(a)=0$. 

Note that $\mathcal N_{\Var }$ and the entire class $\Var\Der$ satisfy the same 
(differential) identities: we prove that within the following statement.

\begin{lemma}\label{lem:LocNilp}
Suppose $f=f(x_1,\dots, x_n)$ is a multilinear identity that 
holds on $A^{(d*)}$ for all $A\in \mathcal N_{\Var }$. 
Then $f$ is a generalized derived identity of $\Var $.
\end{lemma}

\begin{proof}
Consider the free differential $\Var$-algebra $U_n$ generated 
by the set $X=\{x_1,x_2,\dots \}$
with one derivation $d$ modulo defining relations $d^{n}(x)=0$, $x\in X$. 

Then $U_n\in \mathcal N_{\Var }$. As a differential $\Var$-algebra, $U_n$ is a homomorphic 
image of the free magmatic algebra $M\<d^{\omega }X\>$. 
Denote by $I_n$ the kernel of the homomorphism $M\<d^{\omega }X\> \to U_n$.
As an ideal of $M\<d^{\omega }X\>$, $I_n$ is a sum of two ideals: $I_n=T_{\Var}(d^\omega X)+N_n$,
where $N_n$ is generated by $d^{n+t}(x)$, $x\in X$, $t\ge 0$. 
Note that the last relations form a $d$-invariant subset of $M\<d^\omega X\>$, 
so the ideal $N_n$ is $d$-invariant. 

If a polynomial $F\in M\<d^\omega X\>$ belongs to $I_n$ and 
its degree in $d$ is less than $n$ then $F$ belongs to $T_{\Var }(d^\omega X)$. 
If $f=f(x_1,\dots , x_n)$ is a multilinear identity in $U_n^{(d)}$ 
then the image $F$ of $f$ in $M\<d^\omega X\>$ has degree $n-1$ in $d$, 
so $F\in T_{\Var }(d^\omega )$. 
Hence, $f$ is a derived identity of $\Var $. 
\end{proof}

\begin{theorem}\label{thm:WhiteGD->GenDerived}
If a multilinear identity $f$ holds on the variety 
governed by the operad $\Var \circ \GD^!$ then 
$f$ is a generalized derived identity of $\Var $. 
\end{theorem}

\begin{proof}
For every $A\in \mathcal N_{\Var }$ we may construct an algebra in the variety 
$\Var \circ \GD^! $ as follows. Consider the linear space $H$ spanned by elements $x^{(n)}$, $n\ge 0$, with 
multiplication 
\[
 x^{(n)}\star x^{(m)} = \binom{n+m-1}{m} x^{(n+m-1)}, \quad x^{(n)}\ast x^{(m)} = \binom{n+m}{m} x^{(n+m)}.
\]
This is a $\GD^!$-algebra (in characteristic 0, this is just the ordinary polynomial algebra
with $x^{(n)}=x^n/n!$).
Denote $\hat A = A\otimes H$ and define operations $\prec $, $\succ $, and $\cdot $ on $\hat A$ by
\begin{gather}
 (a\otimes f)\succ (b\otimes g) = ab\otimes (f\star g), \label{eq:SuccOtimes}\\ 
 (a\otimes f)\prec (b\otimes g) = ab\otimes (g\star f), \label{eq:PrecOtimes}\\
 (a\otimes f) (b\otimes g) = ab\otimes (f\ast g), \label{eq:ProdOtimes}
\end{gather}
The algebra $\hat A$ obtained belongs to the variety governed by the operad $\Var \circ \GD^! $.
There is an injective map 
\begin{align}
 \Phi:{} & A \to \hat A, \nonumber \\
 & a\mapsto \sum\limits_{s\ge 0} d^s(a)\otimes x^{(s)}, \label{eq:A-to-hatA}
\end{align}
which preserves operations $\prec $, $\succ $, and $\cdot $. Indeed, let us check that $\Phi $ preserves $\prec $. Apply \eqref{eq:A-to-hatA} to 
$a\prec b$, $a,b\in A$:
\begin{multline}\nonumber
(a\prec b)=ad(b) \mapsto \sum\limits_{s\ge 0}  d^s(ad(b))\otimes x^{(s)} 
= \sum\limits_{s,t\ge 0} \binom{s+t}{s} d^s(a)d^{t+1}(b)\otimes x^{(s+t)} \\
= \sum\limits_{s,t\ge 0}  d^s(a)d^{t+1}(b) \otimes x^{(s)}\cdot x^{(t+1)} 
=
\left(\sum\limits_{s\ge 0} d^s(a)\otimes x^{(s)}\right)
\prec
\left(\sum\limits_{t\ge 0} d^t(a)\otimes x^{(t)}\right).
\end{multline}
Since $\Phi $ is injective, $A^{(d*)}$ is isomorphic to a subalgebra of
a $(\Var\circ \GD^! )$-algebra, 
so all 
defining relations of $\Var\circ \GD^! $ hold on $A^{(d)}$.
Lemma \ref{lem:LocNilp} completes the proof.
\end{proof}

It is easy to see that if we forget about the operation $\cdot $ on $\hat A$ then 
we obtain a proof of the following 

\begin{theorem}\label{thm:WhiteNov->Derived}
If a multilinear identity $f$ holds on the variety 
governed by the operad $\Var \circ \Nov$ then 
$f$ is a derived identity of $\Var $. \qed
\end{theorem}

It remains to prove the converse: why all generalized derived identities 
of $\Var $ hold on all $(\Var\circ \GD^!)$-algebras.

\begin{theorem}\label{thm:ConverseGD}
Let $f$ be a  multilinear generalized derived identity of $\Var $. 
Then $f$ holds on all $(\Var\circ \GD^!)$-algebras.
\end{theorem}

\begin{proof}
Let $X=\{x_1,x_2,\dots \}$, and let $f(x_1,\dots , x_n)$ be a non-associative multilinear polynomial 
in three operations of multiplication $\cdot$, $\prec$, and $ \succ$. 
By the construction of the Manin white product, if an identity 
holds on $\Var\<X\>\otimes \GD^!\<X\>$ equipped with operations 
\eqref{eq:SuccOtimes}--\eqref{eq:ProdOtimes} then it holds on all 
$(\Var\circ \GD^!)$-algebras.

By Corollary \ref{eq:ComDerWeight}, $\GD^!\<X\>$ is a subalgebra of $C^{(d*)}$ for 
the free differential commutative algebra $C=\Com\Der\<X,d\>$. Therefore, 
$\Var\<X\>\otimes \GD^!\<X\>$ is a subalgebra of $(\Var\<X\>\otimes C)^{(D*)}$,
where $D(u\otimes f)=u\otimes d(f)$ for $u\in \Var\<X\>$, $f\in C$.
Since $A=\Var\<X\>\otimes C \in \Var$, the identity $f(x_1,\dots, x_n)=0$ 
holds on $A^{(D*)}$, so it holds on $\Var\<X\>\otimes \GD^!\<X\>$.
\end{proof}

Similarly, we have 

\begin{theorem}\label{thm:ConverseNov}
Let $f$ be a  multilinear derived identity of $\Var $. Then $f$ holds on all $(\Var\circ\Nov)$-algebras.
\end{theorem}

\begin{remark}
In this section, we have considered algebras with one binary operation, so the operad $\Var $ has only one generator. 
However, Theorems~\ref{thm:WhiteGD->GenDerived}--\ref{thm:ConverseNov} are straightforward 
to generalize for algebras 
with multiple binary operations, i.e., for an arbitrary binary operad $\Var $. 
In the next sections, we apply these results to Poisson algebras.
\end{remark}

\section{Free special GD-algebra}\label{sec:SpecialGD}

Denote by $\Pois$ the variety of all Poisson algebras, and let $\Pois\Der$ stand 
for the variety of differential Poisson algebras. Theorem \ref{thm:th1} determines 
a functor from $\Pois\Der $ to $\GD$ corresponding to the following 
morphism of operads:
\[
 \begin{aligned}
  \delta :& \GD \to \Pois\Der, \\
   & x_1\circ x_2\mapsto x_1 d(x_2), \\
   & [x_1,x_2]\mapsto \{x_1,x_2\}
 \end{aligned}
\]

Let us say that a $\GD$-algebra $A$ is {\em special} if there exists $V\in \Pois$ 
such that $A$ is a subalgebra of~$V^{(\delta )}$.
Denote by $\SGD$ the variety generated by special $\GD$-algebras, i.e., the class 
of all homomorphic images of all special $\GD$-algebras. 
The main purpose of this section is to describe the free $\SGD$-algebra.

Let $X$ be a linearly ordered set of generators. 
Denote by $\Pois\Der\<X,d\>$ the free Poisson differential algebra with a derivation $d$.
Consider the set 
$d^\omega X =\{d^n(x)\mid n\geq 0, x\in X\}$. 
Let us consider the elements $d^n(x)\in d^\omega X$ as pairs $(n,x)$
and compare them lexicographically.
Then $\Pois\Der\<X,d\>$ is isomorphic to $\Pois\<d^\omega X\>$ as a Poisson algebra.

As a  commutative algebra, $\Pois\Der\< X,d\>$ is isomorphic to 
the symmetric algebra over the space $\Lie \< d^\omega X\>$.
Recall that a linear basis of the free Lie algebra $\Lie\<Y\>$ 
generated by a linearly ordered set $Y$ may be chosen as follows. 
An associative word $U\in Y^*$ is said to be a {\em Lyndon--Shirshov word} if for every presentation $U=U_1U_2$, $U_i\in Y^*$, we have 
$U>U_2U_1$ lexicographically. This definition goes back to \cite{CFL} and \cite{Sh58} 
(see also \cite{BC_BMS}). For every Lyndon--Shirshov word $U$ there is a unique ({\em canonical}) bracketing $[U]$ such that 
$[U]\in \Lie \<Y\>$ and the set of all Lyndon--Shirshov words with canonical bracketing forms a linear basis of $\Lie \<Y\>$.

\begin{definition}
Let $u$ be a monomial in $\Pois\Der\<X,d\>$. Define the {\em weight} 
function $\wt(u)\in \mathbb Z$ by induction 
in the following way:
\[
\begin{gathered}
\wt(x)=-1,\quad x\in X; \\
\wt(d(u)) = \wt(u)+1; \\
\wt(\{u,v\})=\wt(u)+\wt(v)+1; \\
\wt(uv)=\wt(u)+\wt(v).
\end{gathered}
\]
\end{definition}
Since all defining identities of  Poisson algebras 
are homogeneous relative to $\{\cdot,\cdot\}$ and $\cdot $, 
the function $\wt$ is well-defined.

By the general algebra arguments (see, e.g., \cite{MikhSh10}) 
$\SGD\<X\>$ is isomorphic to the subalgebra of 
$\Pois\Der\<X,d\>$ generated by $X$ relative to the operations 
$[u,v]=\{u,v\}$ and $u\circ v = ud(v)$.

Consider a linear map 
\[
\begin{gathered}
\phi : \GD\<X\> \to \SGD\<X\>\subset \Pois\Der\<X,d\>, \\
\phi(x)=x, \quad x\in X, \\
\phi([u,v])=\{\phi(u),\phi(v)\}, 
\quad \phi(u\circ v)=\phi(u) d(\phi(v)). 
\end{gathered}
\]
By Theorem \ref{thm:th1}, $\phi $ is an epimorphism of $\GD$-algebras.

For a monomial $u\in \GD\<X\>$, denote by $m(u)$ the number of Novikov multiplications in $u$ or, 
which is the same, the number of commutative multiplications in $\phi(u)$. 
For a monomial $a\in \Pois\Der\<X,d\>$, denote $D(a)=n$ the total number of derivations
$d$ in $a$. It is easy to see that $D(\phi(u))=m(u)$. 

\begin{lemma}\label{lem:Ejection}
Let $a=x_1\dots x_t [U]\in \Pois\Der\<X,d\>$, 
where $x_i\in X$, $U$ is a Lyndon--Shirshov word in $d^\omega X$, 
$D([U])=t$.
Then there exists $f\in \GD\<X\>$ such that $\phi(f)=a$.
\end{lemma}

\begin{proof}
Let us prove the statement by induction on $D([U])$ and $|U|$, where $|U|$ is the length of $U$ as of a word in $d^\omega X$.

If $D([U])=t=0$ then the desired pre-image may be obtained just by replacing 
$\{\cdot,\cdot\}$ with $[\cdot,\cdot]$. We do not distinguish notations for 
$[U]\in \Lie\<X\>\subset \Pois\Der\<X,d\>$ 
and its pre-image in $\GD\<X\>$.

Suppose $D(U)=1$. If $|U|=1$ then the result is obvious. Let $|U|=2$. Then 
\[
x_1 \{d(y_1),y_2\}=
\{x_1 d(y_1),y_2)\}-\{x_1,y_2\} d(y_1),
\]
so 
$[x_1\circ y_1,y_2]-[x_1,y_2]\circ y_1$ is a desired pre-image.

If $|U|=l>2$ then $[U]=\{[d(y_1)u],[v]\}$ for some words $u$, $|u|=i<l-1$, and $v$, $|v|=j<l$, by the definition of a Lyndon--Shirshov word ($d(y_1)$ is the greatest letter in $U$).
Assume there exists 
\[
 F_{1,j}(x_1,[W])\in \phi^{-1}(x_1 [W])
\]
for all Lyndon--Shirshov words $W$ such that $D(W)=1$, $|W|=j<l$. 
Then 
\[
x_1 [U] = x_1 \{[d(y_1)u],[v]\} = \{x_1 [d(y_1)u], [v]\} - [d(y_1)u]\{x_1,[v]\}.
\]
Therefore, 
\[
 F_{1,l}(x_1, [U]) = [F_{1,i+1}(x_1, [d(y_1)u]), [v]] - F_{1,i+1}([x_1,[v]], [d(y_1)u])
\]
is a desired pre-image of $x_1 [U]$.

Suppose $D(U)=t>1$ and $|U|=1$, i.e., $U=d^t(y)$, $y\in X$. 
Then there exists a pre-image of
$x_1\dots x_t d^t(y)$
\[
 F_{t,1}(x_1,\dots, x_t, d^t(y)) \in \Nov\<X\>\subset \GD\<X\>,
\]
as shown in \cite{DzhLowf}.

Assume $D(U)=t>1$, $|U|=l>1$, and
\[
 F_{i,j}(x_1,\dots, x_i, [W]) \in \phi^{-1} (x_1 \dots x_i [W])
\]
exists for all Lyndon--Shirshov words $W$ 
such that either $i=D(W)<t$ or $D(W)=t$, $j=|W|<l$. 
Since 
\[
 [U] = \{[d^{s}(y)u],[v] \}, \quad 0<s\le t,\ y\in X,
\]
where 
\[
 1\le D(d^s(y)u)=i\le t,\quad D(v)=t-i<t,\quad |d^{s}(y)u|=j<l, \quad |v|=q<l,
\]
we have
\begin{multline}\nonumber
 x_1 \dots x_t [U] 
= x_{i+1}\dots x_t \{x_1 \dots  x_i [d^s(y)u], [v]\}  \\
 - \sum\limits_{k=1}^i x_1 \dots \hat x_k \dots  x_t [d^{s}(y)u] \{x_k,[v]\}, 
\end{multline}
where $\hat x_k$ means that the factor $x_k$ is omitted.
By the inductive assumption, 
there exist
\[
 P\in \phi^{-1}(x_1\dots  x_i [d^s(y)u])=F_{i,j}(x_1,\dots, x_i, [d^s(y)u]) 
\]
and 
\[
 P_k\in \phi^{-1}(x_{i+1} \dots  x_t \overline{\{x_k,[v]\}}),
\]
where 
$\overline{f}$ stands for the linear combination 
of Lyndon--Shirshov words  with canonical bracketing representing a Lie polynomial $f\in \Lie\< d^\omega X\>$.
Obviously, we may choose
\begin{multline}\label{eq:IndPhi}
 \phi^{-1}(x_1 \dots x_t [U]) = 
 F_{t-i,q+1}(x_{i+1},\dots, x_t, \overline{\{z,[v]\}})|_{z=P}  \\
 -\sum\limits_{k=1}^i F_{t-i,q+1} (z,x_1,\dots, \hat x_k, \dots , x_i, [d^s(y)u])|_{z=P_k}
\end{multline}
Here we consider the variable $z$ in the right-hand side of \eqref{eq:IndPhi}
as a new one, and assume $z<x$ for all $x\in X$.
\end{proof}


\begin{theorem}\label{TheoremSGD}
Differential Poisson monomials of weight $-1$ span the special GD-algebra $\SGD\<X\>\subset \Pois\Der\<X,d\>$.
\end{theorem}

\begin{proof}
Obviously,  
$\phi(f)$ is a linear combination of differential Poisson monomials of weight~$-1$ 
for every $f\in \GD\<X\>$.

Conversely, let $a$ be a monomial in $\Pois\Der\<X,d\>$ such that $\wt(a)=-1$.
We have to show that there exists $\phi^{-1}(a)\in \GD\<X\>$.
It is enough to prove it for basic elements of $\Pois\Der\<X,d\>$, 
i.e., elements of the form
\begin{equation}\label{eq:2}
a=[X_1] [X_2]\cdots [X_n],    
\end{equation}
where $X_i$ are  Lyndon--Shirshov words in $d^\omega X$, 
$[\cdot ]$ denotes the canonical bracketing,
$X_1\le \dots \le X_n$. 

Let us proceed by induction on $D(a)$. 
If $D([X_i])=k_i$ then $\wt([X_i])=k_i-1$, so $\wt([X_i])\ge -1$. 
Note that $\wt(a) = \wt([X_1])+\dots + \wt([X_n])=D(a)-n$.
Hence, if $D(a)=0$ then $n=1$ and $a=[X_1]$ is a pure Lie element in $X$
which belongs to $\GD\<X\>$.

Assume that for some $m>0$ all monomials $b\in \Pois\Der\<X,d\>$
such that $\wt(b)=-1$, $D(b)<m$ have pre-images in $\GD\<X\>$.
Suppose $D(a)=m$ then, without loss of generality, we may assume 
$D([X_1])=k>0$, $k\le m$. 
Then there must exist $k$ differential-free factors, 
say $[X_2],\dots , [X_{k+1}]$ of weight~$-1$.
Lemma \ref{lem:Ejection} allows us to replace 
$[X_1] \cdots [X_{k+1}]$ with a new variable $y$, $\wt(y)=-1$, 
find a pre-image 
$f(y) $ of $y [X_{k+2}] \cdots [X_n])\in \GD\<X\cup\{y\}\>$,
and find the desired pre-image of $a$ as a substitute 
$f(y)|_{y\in \phi^{-1}([X_1]\cdots [X_{k+1}])}$.
\end{proof}

In order to calculate $\dim \SGD(n)$, we need to count Poisson differential
monomials 
in $x_1,\dots, x_n$ of weight $-1$. Let us introduce the following function on $n,k\in \mathbb Z$, 
$n,k\ge 1$:
\[
\begin{gathered}
 H(n,1)=1, \\
 H(n,k) = \sum\limits_{1\le d_1<\dots <d_{k-1}\le n-1} \dfrac{1}{d_1 \dots d_{k-1}}, 
 \quad k\ge 2.
\end{gathered}
\]
In particular, $H(n,2)$ is the $n$th partial sum of the harmonic series.

\begin{corollary}
 For the operad $\SGD$, we have 
 \[
  \dim \SGD(n) = \sum\limits_{k=1}^n \dfrac{(n+k-2)!}{(k-1)!} H(n,k).
 \]
\end{corollary}

\begin{proof}
Consider the set $X=\{x_1,x_2,\dots \}$ which is linearly ordered as above:
$x_1<x_2<\dots $.
Let us first calculate the number $L(n,k)$ 
of homogeneous Poisson monomials $u_1\dots u_k\in \Pois(n)$
that split into $k$ Lyndon---Shirshov factors $u_i\in \Lie\<X\>$, 
$u_1\le \dots \le u_k$ lexicographically. 
If $k=1$, 
$L(n,1)=(n-1)!$. If $k>1$ then $u_k$ starts with the greatest letter $x_n$. 
Suppose the length of $u_k$ is $d$, $1\le d\le n-1$. The number of such 
Lyndon---Shirshov words is $(d-1)!\binom{n-1}{d-1}$.
The remaining $n-d$ letters form a product of $k-1$ Lyndon---Shirshov words 
$u_1,\dots , u_{k-1}$. Hence, 
\[
 L(n,k) = \sum\limits_{d=1}^{n-k+1}(d-1)!\binom{n-1}{d-1} L(n-d, k-1)
\]
Therefore, 
\begin{multline}\nonumber
 L(n,k) \\
 = \sum\limits_{d_1,\dots, d_{k-1}} 
 \dfrac{(n-1)!(n-d_1-1)! (n-d_1-d_2-1)! \dots (n-d_1-\dots -d_{k-1}-1)!}
  {(n-d_1)!(n-d_1-d_2)! \dots (n-d_1-\dots -d_{k-1})!} \\
  =
  (n-1)! \sum\limits_{d_1,\dots , d_{k-1}} \dfrac{1}{(n-d_1)(n-d_1-d_2) \dots (n-d_1-\dots -d_{k-1})},
\end{multline}
where the sum ranges over all $(d_1,\dots , d_{k-1})$ such that 
$d_i\ge 1$ and $d_1+\dots + d_{k-1}<n$. It is easy to see that 
\[
 L(n,k)=(n-1)! H(n,k)
\]
for $n,k\ge 2$.

In order to get a differential Poisson monomial of weight $-1$, we have to apply $k-1$ derivations $d$ 
to the letters of a Poisson monomial $u_1\dots u_k$. There are $n$ letters $x_1,\dots, x_n$, 
and we may differentiate every letter more than once. Hence, the number of all possible combinations is
\[
 \binom{n+k-2}{k-1} L(n,k) = \dfrac{(n+k-2)!}{(k-1)!} H(n,k).
\]
\end{proof}

In particular, for $n=1,\dots, 7$ we have the following values:

\begin{tabular}{c|ccccccc}
 $n$ & 1 & 2 & 3 & 4 & 5 & 6 & 7 \\
 \hline 
 $\dim \SGD(n) $ & 1 & 3 & 17 & 130 & 1219 & 13391 & 167656
\end{tabular}

\section{Defining identities of special GD-algebras}\label{sec:Ident}

In this section, we describe the defining relations of the operad $\SGD $
and 
derive examples of {\em special identities}, i.e., those
identities that distinguish $\GD$ and $\SGD$. 

By definition, the operad $\SGD$ 
is the sub-operad of $\Pois \circ \GD^!$ generated by 
$ x_1x_2\otimes x_2\star x_1$ and $\{x_1,x_2\}\otimes x_1\ast x_2$.
 
\begin{proposition}\label{prop:Pois_GD!}
The variety of $(\Pois\circ \GD^!)$-algebras 
consists of linear spaces $A$ equipped with 
four bilinear operations $*$, $\circ$, $[\cdot ,\cdot ]$, and  $[\cdot \succ \cdot ]$
such that $(A, *, [\cdot , \cdot ])$ is a Poisson algebra, $(A,\circ )$ is a Novikov algebra, 
and the following identities hold:
\begin{gather}
[[x\succ  z],y] = [[z\succ  x],y] - [[x, z]\succ y] ,                                   \label{eq:PGD-1}\\
[[y\succ  x],z] = [[y\succ  z],x] + [y\succ  [x,z]] ,                                  \label{eq:PGD-2}\\
(x*y)\circ z = x*(y\circ z),                                                 \label{eq:PGD-21}\\
x\circ (y*z) = (z*x)\circ y + (y*x)\circ z,                                  \label{eq:PGD-22}\\
[(x*z)\succ y]) = [(x\circ z),y] +[(z\circ x),y],                        \label{eq:PGD-3}\\ 
[x,z]\circ y =  [y\succ  x]*z - [(z\circ y),x],                          \label{eq:PGD-4}\\
y\circ [x,z] =[x\succ  z]*y  - [z\succ  x]*y,                                \label{eq:PGD-5}\\
[y\succ (x*z)] = [(x\circ y),z] + [(z\circ y),x],                       \label{eq:PGD-6}\\  
[y\succ z]\circ x + [x\succ z]\circ y = [x\succ (z\circ y)] + [y\succ (z\circ x)],                                  \label{eq:PGD-7}
\end{gather}
\begin{multline}                                   \label{eq:PGD-8}
[(y\circ z)\succ x] + [(x\circ z)\succ y] = x\circ [z\succ y] +  [z\succ y]\circ x   \\
+ [x\succ (y\circ z)] + y\circ [z\succ x] + [z\succ x]\circ y + [y\succ (x\circ z)] .
\end{multline}
\end{proposition}

\begin{proof}
 Apply the computation process similar to that in Examples \ref{exmp:Lie-Nov}, \ref{exmp:As-Nov}
 relative to the following notations for the generators of 
 $\Pois \circ \GD^!$ in $\Pois \otimes \GD^!$:
\[
\begin{gathered}
x_1*x_2=x_1x_2\otimes x_1*x_2, \quad 
x_1\circ x_2 = x_1x_2\otimes x_2\star x_1, \\
[x_1,x_2] = \{x_1,x_2\}\otimes x_1*x_2, \quad 
[x_1\succ x_2] = \{x_1,x_2\}\otimes x_1\star x_2 .
\end{gathered}
\]
Straightforward solution of the corresponding linear algebra problem leads to
the required identities.
\end{proof}

\begin{corollary}\label{cor:SGD_ident}
The T-ideal $T_{\SGD}(X)$ coincides with the intersection 
of $T_{\Pois\circ \GD^!}(X)$ with the space of 
nonassociative polynomials on operations $\circ $ and $[\cdot ,\cdot ]$.
In other words, to get the defining relations of the operad $\SGD $  it is enough to find all those corollaries 
of the identities stated in Proposition \ref{prop:Pois_GD!}
that contain only the operations $\circ $ and $[\cdot, \cdot ]$.
\end{corollary}

According to Corollary \ref{cor:SGD_ident}, defining identities of $\SGD $ 
may be derived from the identities of $\Pois \circ \GD^!$ given by Proposition \ref{prop:Pois_GD!}.
For example, we may use \eqref{eq:PGD-4} to eliminate the operation $\succ $ in  \eqref{eq:PGD-5}. 
As a result, we obtain \eqref{eq:GD-1}.

Finding the complete list of defining relations for the operad $\SGD $ seems to be a hard problem.
We will deduce some of these relations below.

Consider the monomial $[y\succ x]*(z\circ u)$. Apply \eqref{eq:PGD-21} and then \eqref{eq:PGD-4} to get
\[
 [y\succ x]*(z\circ u) = ([y\succ x]*z)\circ u = ([x,z]\circ y + [(z\circ y),x])\circ u.
\]
On the other hand, apply \eqref{eq:PGD-4} to get 
\[
 [y\succ x]*(z\circ u) = [x,(z\circ u)]\circ y + [((z\circ u)\circ y), x].
\]
Therefore, the following identity holds on $\SGD $-algebras:
\begin{equation}\label{eq:S-ident}
    [x_1,x_3\circ x_2]\circ x_4 + ([x_3,x_1]\circ x_2)\circ x_4
    =
    [x_1,(x_3\circ x_4)\circ x_2] + [x_3\circ x_4, x_1]\circ x_2.
\end{equation}
(It is also easy to check in a straightforward way that the corresponding 
identity holds on differential Poisson algebras.)

\begin{proposition}\label{prop:S-ident1}
The identity \eqref{eq:S-ident}
does not hold on the class of all $\GD $-algebras. 
\end{proposition}

\begin{proof}
Without loss of generality 
we may assume the base field $\Bbbk $ is algebraically closed.
We will find an example of a $\GD $-algebra which does not satisfy \eqref{eq:S-ident}
by means of algebraic geometry arguments using the approach of \cite{Shaf}. 

Let us fix a structure of a Lie algebra $\mathfrak g$ on a $n$-dimensional linear space
spanned by $x_1,\dots, x_n$. Namely, we fix a multiplication table of this 
Lie algebra 
\[
[x_i,x_j] = \sum\limits_{k=1}^n c^{ij}_k x_k,\quad c^{ij}_k\in \Bbbk .
\]
Consider the class of all Novikov products $\circ $ on the space 
$\mathfrak g$
satisfying \eqref{eq:GD-1} as a Zariski closed 
set $\mathcal V(\mathfrak g)$ in the $n^3$-dimensional affine space assuming 
\[
x_i\circ x_j = \sum\limits_{k=1}^n a^{ij}_k x_k,\quad a^{ij}_k\in \Bbbk .
\]
A point $(a^{ij}_k)_{i,j,k=1,\dots , n} \in \Bbbk^{n^3}$ belongs to 
$\mathcal V(\mathfrak g)$
if and only if 
\begin{gather}
\sum\limits_{k=1}^n \big (a^{ij}_k a^{kp}_{l} - a^{ip}_{k} a^{kj}_{l} \big )=0,
  \label{eq:VarGD1} \\
\sum\limits_{k=1}^n \big ( 
a^{ij}_k a^{kp}_l - a^{ik}_l a^{jp}_k - a^{ji}_k a^{kp}_l + a^{jk}_l a^{ip}_k \big) = 0,
  \label{eq:VarGD2} \\
\sum\limits_{k=1}^n \big (
 c^{ik}_l a^{jp}_k - c^{pk}_l a^{ji}_k + c^{ji}_k a^{kp}_l - c^{jp}_k a^{ki}_l - c^{ip}_k a^{jk}_l 
\big) = 0
   \label{eq:VarGD3}
\end{gather}
for all $i,j,p,l=1,\dots, n$.
These relations represent the defining identities of GD-algebras.

Denote by $I(\mathfrak g)$ the ideal in $\Bbbk[a^{ij}_k \mid i,j,k=1,\dots , n]$
generated by the left-hand sides of \eqref{eq:VarGD1}--\eqref{eq:VarGD3}.
If we assume that \eqref{eq:S-ident} holds for all GD-algebras in the class 
$\mathcal V(\mathfrak g)$ then the corresponding polynomials
\begin{equation}\label{eq:S-identCoord}
    \sum\limits_{k,l=1}^n \big (c^{ik}_l a^{jp}_k a^{lq}_s - c^{ij}_l a^{lp}_k a^{lq}_s - c^{il}_s a^{jq}_k a^{kp}_l + c^{il}_k a^{jq}_l a^{kp}_s \big ),
    \quad 
    i,j,p,q,s=1,\dots, n,
\end{equation}
belong to the radical of $I(\mathfrak g)$.

For dimensions $n=2,3$ it is possible to compute the radical 
of $I(\mathfrak g)$ in a straightforward way by means of the computer 
algebra system Singular \cite{Sing}. These computations show 
that for every 2-dimensional Lie algebra $\mathfrak g$ all polynomials 
\eqref{eq:S-identCoord} belong to $\sqrt{I(\mathfrak g)}$.

However, even for the 3-dimensional Heisenberg algebra $\mathfrak h_3$
spanned by $x_1,x_2,x_3$, $[x_1,x_2]=x_3$, $x_3$ is central, 
there exist polynomials of the form \eqref{eq:S-identCoord} that do not 
belong to $\sqrt{I(\mathfrak h_3)}$. 
\end{proof}

Let us find one more (independent) special identity for $\GD $-algebras. 
Consider the expression 
\[
 ([x\succ (z\circ y)]+[y\succ (z\circ x)])*u. 
\]
Apply \eqref{eq:PGD-7} first, then \eqref{eq:PGD-21} and commutativity of $*$:
\begin{multline}\nonumber
 ([x\succ (z\circ y)]+[y\succ (z\circ x)])*u = ([x\succ z]\circ y+[y\succ z]\circ x )*u \\
 =([x\succ z]*u)\circ y + ([y\succ z]*u)\circ x. 
\end{multline}
Relation \eqref{eq:PGD-4} and right commutativity of $\circ $ imply
\begin{multline}\nonumber
 ([x\succ (z\circ y)]+[y\succ (z\circ x)])*u =
 2([x, u]\circ x)\circ y + [(u\circ x), z]\circ y + [(u\circ y), z]\circ x .
\end{multline} 
On the other hand, it follows from \eqref{eq:PGD-4} that 
\begin{multline}\nonumber 
 ([x\succ (z\circ y)]+[y\succ (z\circ x)])*u \\
 = [(z\circ y), u]\circ x + [(u\circ x), (z\circ y)] + [(z\circ x), u]\circ y + [(u\circ y), (z\circ x)] .
\end{multline}
Finally, using \eqref{eq:S-ident} we obtain
\begin{multline}\label{eq:S-ident2}
 2[x_4, (x_3\circ x_2)\circ x_1] = [x_4\circ x_1, x_3\circ x_2] + [x_4\circ x_2, x_3\circ x_1] \\
 +[x_4, x_3\circ x_2]\circ x_1+[x_4, x_3\circ x_1]\circ x_2 + [x_3,x_4\circ x_1]\circ x_2+ [x_3,x_4\circ x_2]\circ x_1.
\end{multline}

\begin{proposition}\label{prop:S-ident2}
 The identity \eqref{eq:S-ident2} does not hold on the class of all $\GD$-algebras 
 satisfying \eqref{eq:S-ident}.
\end{proposition}
 
\begin{proof}
Let us consider the class of algebras 
$(V,\circ, [\cdot ,\cdot ])$ satisfying the following identities:
\begin{equation}\label{eq:GD-3nilp}
\begin{gathered}
x_1\circ (x_2\circ x_3)=(x_1\circ x_2)\circ x_3=0,\\
[x_1,x_2]\circ x_3= 0, \\
[x_1,x_2\circ x_3] - [x_3,x_2\circ x_1] = x_2\circ [x_1,x_3], \\
[x_1,x_2]=-[x_2,x_1], \\
[x_1,[x_2,x_3]]=0.
\end{gathered}
\end{equation}
Obviously, all these algebras belong to $\GD $ and 
 \eqref{eq:S-ident} holds.
If we assume that  \eqref{eq:S-ident2} 
holds on all these algebras 
then 
\begin{equation}\label{eq:S2-nilp}
[x_1\circ x_2, x_3\circ x_4]  + [x_1\circ x_4, x_3\circ x_2] = 0    
\end{equation}
has to be a corollary 
of \eqref{eq:GD-3nilp}.
If $\mathrm{char}\,\Bbbk \ne 2$, relation \eqref{eq:S2-nilp} may be transformed to 
\begin{equation}\label{eq:S2-nilp2}
x_1\circ [x_2\circ x_3, x_4]=0
\end{equation}
modulo \eqref{eq:GD-3nilp}.
However, if we write down all those corollaries of 
\eqref{eq:GD-3nilp} that contain two operations $\circ $ and one operation $[\cdot,\cdot ]$
then we see that \eqref{eq:S2-nilp2} does not belong to this list. 
\end{proof}
 
\begin{remark}
 We do not suppose that \eqref{eq:S-ident} and \eqref{eq:S-ident2} exhaust all special identities 
 of $\GD$-algebras. Finding the complete list of special identities (or at least determining whether it is finite)
 is an interesting problem.
\end{remark}

\begin{remark}
If $(A,\circ )$ is a Novikov algebra then 
$(A,\circ, [\cdot, \cdot ])$ with 
\[
[a,b]=a\circ b - b\circ a
\]
is a GD-algebra \cite{GD83}. It is easy to check that \eqref{eq:S-ident} and \eqref{eq:S-ident2}
hold for such a GD-algebra. We conjecture that all
GD-algebras obtained in this way are special.
\end{remark}

\section*{Acknowledgements}
The first author was supported by the Program of fundamental scientific researches of the Siberian Branch of 
Russian Academy of Sciences, I.1.1, project 0314-2016-0001.

\end{document}